\documentclass [10pt,reqno]{amsart}

\newlength{\myhmargin}  
\usepackage{amsmath,amssymb,amsthm,amsfonts,amscd,flafter,
graphicx,verbatim,pinlabel,mathrsfs,enumitem}
\usepackage[all]{xy}
\usepackage{epstopdf}
\usepackage[colorlinks=false]{hyperref}
\usepackage[all]{hypcap}
\usepackage{hyperref}
\usepackage{array}

\usepackage{pstricks}
\usepackage[latin1]{inputenc}

\newtheorem{theorem}{Theorem}[section]
\newtheorem{lemma}[theorem]{Lemma}

\newtheorem{prop}[theorem]{Proposition}
\newtheorem{proposition}[theorem]{Proposition}

\newtheorem{conjecture}[theorem]{Conjecture}

\theoremstyle{definition}
\newtheorem{definition}[theorem]{Definition}
\newtheorem{remark}[theorem]{Remark}

\renewcommand{\epsilon}{\varepsilon}

\DeclareMathAlphabet{\mathpzc}{OT1}{pzc}{m}{it}
\usepackage{mathrsfs}

\newcommand{\Z}{\mathbb{Z}}

\newcommand{\RP}{\mathbb{R}\mathbf{P}}

\renewcommand{\sl}{\mathfrak{sl}}

\newcommand{\galg}{\mathfrak{g}}

\renewcommand{\bar}{\overline}

\newcommand{\cab}{\text{cab}}
\newcommand{\cA}{\mathcal{A}}

\newcommand{\cC}{\mathcal{C}}

\newcommand{\X}{\mathcal{X}}
\newcommand{\Xcol}{\mathcal{X}_{\rm col}}

\newcommand{\Khcol}{\text{\em Kh}_{\rm col}}

\DeclareMathOperator{\cone}{Cone}

\DeclareMathOperator{\Sq}{Sq}

\begin{document}
\thispagestyle{empty}
\parindent0em
\title[A Khovanov stable homotopy type for colored links]{A Khovanov stable homotopy type for colored links}
\author[Andrew Lobb]{Andrew Lobb}
\address{Department of Mathematical Sciences\\ Durham University \\ United Kingdom}
\email{andrew.lobb@durham.ac.uk}

\author[Patrick Orson]{Patrick Orson}
\address{Department of Mathematical Sciences\\ Durham University \\ United Kingdom}
\email{patrick.orson@durham.ac.uk}

\author[Dirk Sch\"utz]{Dirk Sch\"utz}
\address{Department of Mathematical Sciences\\ Durham University \\ United Kingdom}
\email{dirk.schuetz@durham.ac.uk}

\thanks{The authors were partially supported by the EPSRC Grant EP/K00591X/1.}

\begin {abstract}
We extend Lipshitz-Sarkar's definition of a stable homotopy type associated to a link $L$ whose cohomology recovers the Khovanov cohomology of $L$.  Given an assignment $c$ (called a \emph{coloring}) of positive integer to each component of a link $L$, we define a stable homotopy type $\Xcol(L_c)$ whose cohomology recovers the $c$-colored Khovanov cohomology of $L$.  This goes via Rozansky's definition of a categorified Jones-Wenzl projector $P_n$ as an infinite torus braid on $n$ strands.

We then observe that Cooper-Krushkal's explicit definition of $P_2$
also gives rise to stable homotopy types of colored links (using the restricted palette $\{ 1 , 2 \}$), and we show that these coincide with $\Xcol$.  We use this equivalence to compute 
the stable homotopy type of the $(2,1)$-colored Hopf link and the $2$-colored trefoil.  Finally, we discuss the Cooper-Krushkal projector $P_3$ and make a conjecture of $\Xcol(U_3)$ for $U$ the unknot.
\end {abstract}

\maketitle

\section{Introduction}
\subsection{Categorification}
Given a semisimple Lie algebra $\galg$ and a link $L \subset S^3$ in which each component of $L$ is decorated by an irreducible representation of $\galg$, the Reshetikhin-Turaev construction returns an invariant of that link that can, in principle, be computed combinatorially from any diagram of $L$.  The standard example is the Jones polynomial, which arises from decorating all components with the fundamental representation $V=V^1$ of $\sl_2$ (here the superscript $1$ on the representation refers to the highest weight of $V$ being $1$).  There are then two obvious first directions in which one can generalize.

On the one hand, one might vary the Lie algebra and consider instead $\sl_n$, but still with the fundamental representation of $\sl_n$.  Each invariant obtained this way is a 1-variable specialization of the 2-variable HOMFLYPT polynomial, and satisfies an oriented skein relation, which yields the benefit of easy computability.

On the other hand, one might stick with $\sl_2$, but vary the irreducible representation.  There is one irreducible $(n+1)$-dimensional representation $V^n$ (of highest weight $n$) for each $n \geq 1$.  Decorating with $V^n$ gives rise to the so-called $n$-colored Jones polynomial.  The colored Jones polynomials no longer satisfy such pleasant skein relations, but they are powerful - for example giving rise to $3$-manifold invariants (also called Reshetikhin-Turaev invariants or, in another form, Turaev-Viro invariants).

Both the $\sl_n$ polynomials and the colored Jones polynomials admit categorifications - that is, they can be exhibited as the graded Euler characteristic of bigraded cohomology theories.  In the case of $\sl_n$, this is Khovanov-Rozansky cohomology \cite{khr1}.  In the case of the colored Jones polynomial there are constructions due to many authors, some inequivalent, although the two we shall be considering in fact give isomorphic cohomologies.  The first is due to Rozansky \cite{rozansky}, and the second due to Cooper and Krushkal \cite{cooperkrushkal}.  
In both cases, the fundamental representation of $\sl_2$ gives Khovanov cohomology \cite{kh1}.

\subsection{Spacification}
Recently it has been shown that Khovanov cohomology admits a \emph{spacification}, that is, for any link there is a stable homotopy type $\X(L)$ whose cohomology gives Khovanov cohomology (the \emph{bi}grading of Khovanov cohomology is recovered from a splitting of $\X(L)$ into wedge of spaces indexed by the integers).  This is work due to Lipshitz and Sarkar \cite{LipSarKhov}.  We note that the term `spacification' is not yet well-defined, since it is unclear exactly what properties one should require of it (for example: should just taking a wedge of the Moore spaces determined by the cohomology count as a spacification?)  Nevertheless, we find it a convenient shorthand for now.

It is a natural question if other Reshetikhin-Turaev invariants admitting categorifications can further be spacified.  In the $\sl_n$ case, work by two of the authors with Dan Jones \cite{JLS} has constructed an $\sl_n$ stable homotopy type given the input of a \emph{matched} knot diagram.  There is good evidence that this stable homotopy type should be diagram-independent.  For $n=2$ it agrees with the stable homotopy type due to Lipshitz-Sarkar.

The case of the colored Jones invariants is, in a sense, a little easier.  In particular, Rozansky's categorification admits spacification.  In the case of the $c$-colored unknot whose categorification is, in Rozansky's construction, the stable limit of the Khovanov cohomology of $c$-stranded torus links as the number of twists goes to infinity, this has been observed by Willis \cite{Willis}, whose paper appeared on the arXiv while this one was being written.  The case of a $c$-colored link in general is no harder, and in fact Rozansky has already taken care of the difficult work.

Since the Cooper-Krushkal and the Rozansky categorifications are equivalent, the natural expectation is that one can lift the Cooper-Krushkal categorification to a spacification equivalent to the Rozansky spacification.  This turns out to be straightforward in the $2$-colored case, but at least the more obvious attempt fails in the $3$-colored case, as we discuss later.

\subsection{Computational results}
\label{subsec:comp}
We shall define a stable homotopy type $\Xcol(L_c)$ where $L_c$ is a framed link with a coloring $c$ of its components by positive integers.  Picking the coloring $1$ for each component returns the stable homotopy type $\Xcol(L_1)$, a grading-shifted version of Lipshitz-Sarkar's stable homotopy type $\X(L)$.

We make some computations for certain links and colorings in Section \ref{sec:examples}.  Already in the simplest case these show interesting behaviour: the link with the lowest positive crossing number is the Hopf link and the first coloring which has not yet been considered by Lipshitz-Sarkar is where one component is colored with $2$ and the other with $1$.  The tail of the colored Khovanov cohomology of the $(2,1)$-colored Hopf link agrees with the tail of the colored Khovanov cohomology of the $(2,1)$-colored $2$-component unlink.  Nevertheless, we observe that even these tails can be distinguished by the stable homotopy type.

Although we are not yet able to compute fully the stable homotopy type of the $3$-colored unknot we make a conjecture based on some partial computations.  This conjecture is interesting because its truth would imply that the periodicity of the tail of the stable homotopy type of a colored link (even in the case of the $3$-colored unknot) can be longer than the periodicity of the tail of its cohomology.

\subsection{Plan of the paper}
In Section \ref{sec:2app} we first observe that we can combine Rozansky's insight with the work of Lipshitz-Sarkar \cite{LipSarKhov}.  This combination is straightforward and yields a stable homotopy type of a framed colored link whose cohomology recovers colored Khovanov cohomology.  Secondly, we give ourselves a framework in which to make computations.  For this it makes more sense to use the Cooper-Krushkal categorification which, at least in the case of colors $2$ and $3$, is entirely explicit.  We define what we mean by a lift of the Cooper-Krushkal categorification to a spacification and show that any such lift gives the same stable homotopy type as that arising from Rozansky's construction.

In Section \ref{sec:CK_spacify_construction}, we construct such a lift of the Cooper-Krushkal categorification for colorings taken from the restricted palette $\{ 1 , 2 \}$.  The case of 3-colored cannot be made to work in the way that one might expect (there is an explicit obstruction to this).  Finally, in Section \ref{sec:examples} we make computations as already discussed in Subsection \ref{subsec:comp}.  At the end of this section we give a discussion of the Cooper-Krushkal 3-colored case.


\section{Two approaches to a colored stable homotopy type}
\label{sec:2app}
The colored Jones polynomial is an invariant of framed links $L$ in which each component of $L$ has been assigned a \emph{color}, or in other words a positive integer weight.  We write the color of a component $k$ of $L$ as $c(k)$, and often keep track of the coloring as a subscript $L_c$.

To compute the polynomial one takes a diagram of $L_c$ in which the self-writhe of each component is equal to its framing.  Then one replaces each component $k$ by $c(k)$ parallel copies following the blackboard framing.  Finally, one places on each component a \emph{Jones-Wenzl projector}.  This projector is an element of the relevant Temperley-Lieb algebra, with coefficients in rational functions of $q$.  Finally, one applies the Kauffman bracket, and obtains an element of $\Z[[q,q^{-1}]$ by expanding in powers of $q$.

The Jones-Wenzl projector is idempotent and satisfies turnback-triviality.  It turns out that these two universal properties are enough to determine it completely.  The Jones-Wenzl projector should in principle lift, in a categorification of the colored Jones polynomial, to a complex in Bar-Natan's tangles-and-cobordisms category \cite{bncob}, satisfying idempotence and turnback-triviality up to chain homotopy equivalence.  Cooper-Krushkal \cite{cooperkrushkal} and Rozansky \cite{rozansky} give ways of achieving such a lift.  
Cooper-Krushkal proceed explicitly and give a categorified projector that they define inductively, while Rozansky realizes the categorified projector as a limit of the complexes associated to torus braids.  It is surprising that the latter approach had apparently not been considered even at the decategorified level until Rozansky's insight!  As observed by Cooper-Krushkal, categorified universal properties imply that the two competing categorifications give identical cohomological 
groups.



\subsection{Grading and other conventions}
We note that there is a discrepancy in the grading conventions between the original paper of Khovanov's \cite{kh1}, Rozansky's torus braids paper \cite{rozansky}, and Cooper-Krushkal's paper \cite{cooperkrushkal}.  We apologize for possibly adding to the confusion.  We shall essentially work with the bigrading conventions used by Bar-Natan \cite{bncob} up to an overall shift.  The overall shift makes it easier to treat the colored Khovanov cohomology as an invariant of a colored framed link, with no choice of orientation.  The convention is depicted in Figure \ref{fig:grading_conventions}.

\begin{figure}
	\resizebox{6cm}{1.4cm}{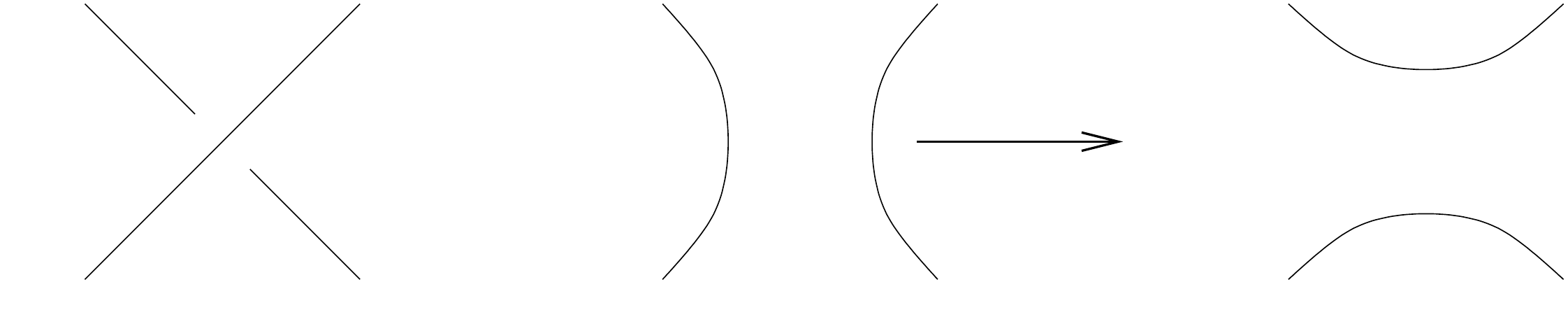}
	\caption{We follow the grading conventions as depicted in the complex that we associate to a single crossing.  The complex is supported in cohomological degrees $\pm 1/2$, and a quantum grading shift is applied.  The differential increases the cohomological degree by $1$ and preserves the quantum grading.}
	\label{fig:grading_conventions}
\end{figure}

With these conventions, the Khovanov complex is invariant up to bigraded homotopy equivalence under the second and third Reidemeister moves, but it is only invariant up to an overall shift under the first Reidemeister move.  Hence it becomes a chain homotopy invariant of framed links (where the framing is given by the blackboard-framing of a diagram).


\subsection{Rozansky spacification}
Rozansky \cite{rozansky} has given an approach to colored Khovanov cohomology that expresses the $c$-colored cohomology of a link $L$ as the limit of the Khovanov cohomologies of an $c$-strand cable of $L$ in which one puts an increasing number of twists.  The stabilization of the cohomology was observed earlier by Sto\v{s}\'ic in the case of $L$ being the unknot, which amounts to the stabilization of the cohomology of the $(p,c)$-torus link as $p \rightarrow \infty$.

\begin{figure}
	\resizebox{10cm}{6cm}{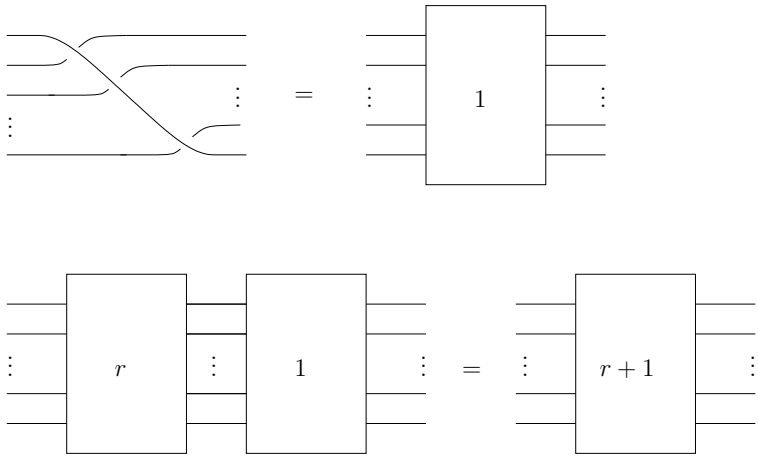}
	\caption{This shows inductively what is meant by twisting $r$ times positively on an $n$-stranded braid.}
	\label{fig:torus_braid}
\end{figure}

We now summarize the construction.  In Figure \ref{fig:torus_braid}, we describe what is meant by twisting $r$ times on an $n$-stranded braid.  We write this braid as $B_{r,n}$.  To each such braid, Bar-Natan's construction \cite{bncob} associates a complex which we shall denote $\langle B_{r,n} \rangle$.  In this complex, each `cochain group' is a vector of tangle smoothings, each such smoothing coming with a quantum degree.  
We shall apply a bigrading shift to this complex so that the resolution which is the identity braid group element is in cohomological degree $0$ and comes with quantum degree shift $0$.  We write the shifted complex as $h^{r(n-1)/2}q^{r(n-1)/2}\langle B_{r,n}\rangle$ where the the exponents of $h$ and $q$ denote cohomological and quantum degree shifts respectively. Note that all other resolutions lie in positive cohomological degrees.

For each $r \geq 1$ there is a map of complexes
\[F_r : (hq)^{rn(n-1)/2}\langle B_{rn,n} \rangle \longrightarrow (hq)^{(r-1)n(n-1)/2} \langle B_{(r-1)n,n} \rangle {\rm ,} \]
given by taking $F_1$ to be the identity in cohomological degree $0$, and then defining $F_r$ to be the tensor product of $F_1$ with the identity on $(hq)^{(r-1)n(n-1)/2}\langle B_{(r-1)n,n} \rangle$.

Rozansky shows that for large $r$ the cone complex $\cone(F_r)$ is homotopy equivalent to a complex in which each smoothing that appears has high cohomological and quantum degrees.  For our purposes, we are mainly interested in the quantum degree; we have

\begin{prop}[Theorem 4.4 \cite{rozansky}]
	\label{prop:roz_cone_limit}
	The cone $\cone(F_r)$ is homotopy equivalent to a complex made up of circleless smoothings, where each such smoothing is shifted in quantum degree by at least $2n(r-1) + 1$.
\end{prop}

The precise form of the quantum degree shift is unimportant for us, rather we note that it increases at least linearly with $r$. 

\begin{definition}
Let $D_c$ be an unoriented link diagram in which each component is colored by a positive integer weight (we write the coloring by weights as $c$), and each component $k$ carries a basepoint.  Let the diagram $D^r_c$ be given by the blackboard-framed $c$-stranded cable of $D_c$ in which each component $k$ receives $c(k)r$ positive twists at the basepoint.
\end{definition}

\begin{definition}
	Let
	\[ G_r : (hq)^{\sum_k rc(k)(c(k)-1)/2} \langle D^{r}_c\rangle  \rightarrow (hq)^{\sum_k (r-1)c(k)(c(k)-1)/2} \langle D^{r-1}_c \rangle  \]
	be induced by the tensor product of the maps $F_r$ at each basepoint.
\end{definition}

\begin{lemma}
	\label{lem:too_much_detail_probably}
It follows from Proposition \ref{prop:roz_cone_limit} that, for fixed $j$ and for all large enough $r$, the map of cohomologies
\[ H^{i,j}((hq)^{\sum_k rc(k)(c(k)-1)/2} \langle D^{r}_c\rangle ) \rightarrow H^{i,j}((hq)^{\sum_k (r-1)c(k)(c(k)-1)/2} \langle D^{r-1}_c \rangle )\]
induced by $G_r$ is an isomorphism.
\end{lemma}

\begin{proof}There is more than one way to see this.  For example, label the components $k_1, \ldots, k_s$
	and write
	\[ e_i = \sum_{j=1}^{j=i-1} (r-1)c(k_j)(c(k_j)-1)/2 +  \sum_{j=i}^{j=s} rc(k_j)(c(k_j)-1)/2{\rm ,}\]
	and denote by $D^{k_i}_c$ the result of taking the $c$-cable of $D$ and adding $rc$ twists at the basepoints of $k_i, \ldots, k_s$ and $(r-1)c$ twists at the basepoints of $k_1, \ldots, k_{i-1}$.
	Then we can write
	\[ G_r = \overline{F^{k_s}_r} \circ \cdots \circ  \overline{F^{k_1}_r} \]
where
\[ \overline{F^{k_i}_r} : (hq)^{e_i} \langle D^{k_i}_c\rangle  \rightarrow (hq)^{e_{i+1}} \langle D^{k_{i+1}}_c \rangle \]
is induced by $F_r$ at a chosen basepoint.  The cone $\cone(\overline{F^{k_i}_r})$ is homotopy equivalent to a complex made up of the tensor product of three Bar-Natan complexes of tangles.  Namely
\begin{itemize}
	\item A complex of circleless smoothings at the chosen basepoint whose quantum degree increases linearly with $r$.
	\item At the other basepoints, the complexes $(hq)^{rn(n-1)/2}\langle B_{rn,n} \rangle$.  After circle removal, these consist of circleless smoothings each in a non-negative quantum degree.  
	This can be seen by observing that the identity braid is in cohomological and quantum degree $0$.  Smoothings in cohomological degree $d$ differ from the identity braid by exactly $d$ surgeries and so contain at most $d-1$ circles.
	\item A complex independent of $r$ arising from the Bar-Natan complex of the diagram away from the basepoints.
\end{itemize}
Finally we recall the homological algebra fact that
\[ \cone(k \circ l) = \cone(\Sigma^{-1}\cone(k) \rightarrow \cone(l))\]
for maps of complexes $k: C \rightarrow C'$, $l: C'' \rightarrow C$.  This implies that $\cone(G_r)$ can be represented by circleless smoothings such that the minimal quantum degree among them increases at least linearly with $r$.
\end{proof}

Hence we can make the following definition.

\begin{definition}
\label{defn:roz_colored_cohom}
For fixed $j$, the $c$-colored Khovanov cohomology of the diagram $D$ framed by the component-wise writhe is defined to be the group
\[ \Khcol^{i,j}(D_c) = H^{i,j}((hq)^{\sum_k rc(k)(c(k)-1)/2} \langle D^{r}_c \rangle )\]
for sufficiently large $r$.
\end{definition}

Independence of the cohomology under Reidemeister moves II and III and under choice of basepoints follows immediately from the independence under Reidemeister moves II and III of standard Khovanov cohomology.  The fact that a suitable Euler characteristic of the cohomology agrees with the $c$-colored Jones polynomial of $D$ is due to Rozansky.

Since $H^{i,j}((hq)^{\sum_k rc(k)(c(k)-1)/2} \langle D^{r}_c \rangle)$ is simply a grading-shifted version of the usual Khovanov cohomology of $D_c^r$, the construction of Lipshitz-Sarkar gives rise to a stable homotopy type $\bar{\X}^j(D^{r}_c)$ realizing it by the (suitably shifted) singular cohomology groups.

Furthermore, observe that the map $G_r$ is induced by quotienting out a subcomplex generated by standard generators of the Khovanov complex.  This subcomplex corresponds to an upward-closed subcategory of the framed flow category associated by Lipshitz-Sarkar to $D_c^r$.  It follows that $G_r$ is induced by a map
\[ g_r : \bar{\X}^j(D^{r-1}_c) \longrightarrow \bar{\X}^j(D^{r}_c) {\rm .}\]
Since for all sufficiently large $r$, $g_r$ gives an isomorphism on cohomology, Whitehead's theorem implies that for sufficiently large $r$, $g_r$ is a homotopy equivalence.

\begin{definition}
\label{defn:roz_color_hom_type}
We can now define the colored stable homotopy type for fixed $j$ to be
\[ \Xcol^j(D_c) = \bar{\X}^j(D^{r}_c) \]
for sufficiently large $r$.  In other words, this is the homotopy colimit of the directed system of maps $g_r$.
\end{definition}

The invariance of this stable homotopy type under choice of basepoints and under Reidemeister moves II and III follows from the invariance of the Lipshitz-Sarkar homotopy type under Reidemeister moves II and III.

\begin{remark}
	Willis \cite{Willis} gave Definition \ref{defn:roz_color_hom_type} in the case that $D$ is the unknot and gave an independent argument that the limit of the system $g_r$ exists.
\end{remark}

\begin{remark}
	\label{rem:R1}
	We note that Definition \ref{defn:roz_color_hom_type} implies that the framing of the link components only affects the colored stable homotopy type up to an overall shift in bigrading, as is the case for the colored Khovanov cohomology.  This is because the blackboard-framed $c$-cable of a $1$-crossing Reidemeister $1$-tangle is equivalent to a full twist in a $c$-stranded braid by a sequence of Reidemeister moves involving $c$ Reidemeister I moves.  Reidemeister moves preserve the stable homotopy type according to Lipshitz-Sarkar, but Reidemeister I moves introduce a shift (with our grading conventions).
\end{remark}

\subsection{Cooper-Krushkal spacification}
\label{subsec:CK-spacify}
In this subsection we give the properties that one might expect of a spacification based on the Cooper-Krushkal categorification.  These properties are enough to imply that any such spacification is stably homotopy equivalent to the Rozansky spacification, as is verified in Subsection \ref{subsec:Roz_equiv_CK}.  The construction of such spacifications is, however, not straightforward, and we leave discussion of these to Section \ref{sec:CK_spacify_construction}.

Suppose that for each $n \geq 1$, $P_n$ is a complex of $(n,n)$-tangle smoothings in the sense of Bar-Natan \cite{bncob}, such that each $P_n$ is a \emph{universal projector} by Definition 3.1 of \cite{cooperkrushkal}.  Cooper-Krushkal have given a way of constructing such universal projectors.  We note that a part of their definition of $P_n$ is that the identity $n$-braid smoothing appears only once and in degree $(0,0)$, and that the quantum and cohomological degrees of every smoothing in the complex are non-negative.

Suppose that $T$ is a tangle diagram in the plane punctured by $k$ discs with $2n_i$ ordered boundary points on the $i$th disc.  Then we may define the Khovanov cochain complex (of free abelian groups) $\langle T_P \rangle$ by taking the tensor product of the Bar-Natan complex $\langle T \rangle$ and $P_{n_i}$ for $i=1, \ldots, k$ in the obvious way.

\begin{definition}
	A \emph{Cooper-Krushkal framed flow category} (C-Kffc) is a choice of finite-object framed flow category (see \cite{LipSarKhov} for definition and references) $\cC(T_P)$ refining the Khovanov cochain complex $\langle T_P \rangle$ for each such $T$.  Choosing a particular crossing of the tangle $T$ we write $T^0$ and $T^1$ for the $0$- and $1$-resolutions of that crossing.  We require that the standard generators corresponding to the subcomplex $\langle T^1_P \rangle$ (resp. the quotient complex $\langle T^0_P \rangle$) correspond to upwards closed (resp. downwards closed) framed flow subcategories of $\cC(T_P)$ such that the associated CW-complex is stably homotopy equivalent to $\vert \cC(T^1_P) \vert$ (resp. $\vert \cC(T^0_P) \vert$).
	
	
	Furthermore, if we denote by $T^{\rm id}$ the tangle diagram produced by filling the $k$th boundary disc of $T$ with the identity $n_k$-braid, then $\langle T^{\rm id}_P \rangle$ is naturally a quotient complex of $\langle T_P \rangle$ generated by standard generators of $\langle T_P \rangle$.  We require this quotient complex to correspond to a downward closed subcategory of $\cC(T_P)$ with associated CW-complex stably homotopy equivalent to $\vert \cC(T^{\rm id}_P) \vert$.
\end{definition}

\begin{remark}
We can restrict this definition if we like to certain values of $n$.  In particular in this paper we give a genuine C-Kffc only for the color $n=2$.  For the color $n=3$ we may slightly alter the definition of a C-Kffc, to arrive at a framed flow category spacifying a cohomology theory that has its graded Euler characteristic a non-standard normalization of the $3$-colored Jones polynomial.  If we insist on the standard normalization we run into difficulties.  We discuss this in Section \ref{sec:CK_spacify_construction}.
\end{remark}

\begin{remark}
	\label{rem:handy_little_remark}
	We note that the condition that a C-Kffc assigns a \emph{finite-object} framed flow category is equivalent to the condition that the minimal quantum degree of the circleless smoothings in the $i$th cochain group of $P_n$ tends to infinity as $i \rightarrow \infty$.  Although this is true for the explicit examples of universal projectors constructed by Cooper-Krushkal, it is not required by them axiomatically.
\end{remark}

\subsection{The equivalence}
\label{subsec:Roz_equiv_CK}

We shall next see that C-Kffc's give rise to the same stable homotopy types as does $\Xcol^j$.  More precisely, let $D$ be a link diagram framed by the component-wise writhe with each component $k$ having a basepoint, and each being colored by a positive integer weight $c(k)$.  We write $D^\cab$ for the tangle formed by cutting $D$ open at each basepoint and then taking the blackboard-framed $c$-cable.  
Then we can consider the Bar-Natan cochain complex of free abelian groups formed by tensoring in $P_{c(k)}$ in the obvious way.  This cochain complex is the Cooper-Krushkal complex that categorifies the colored Jones polynomial of $D$, and it is refined by the framed flow category $\cC(D^\cab_P)$.  Writing $\cC^j(D^\cab_P)$ for the part of this framed flow category in quantum degree $j$, we have the following result.

\begin{proposition}
	\label{prop:equiv}
	With the diagram $D$ as above we have
	\[\Xcol^j(D_c) \simeq |\cC^j(D^\cab_P)| {\rm .}\]
\end{proposition}

\begin{figure}
	\resizebox{5cm}{2.5cm}{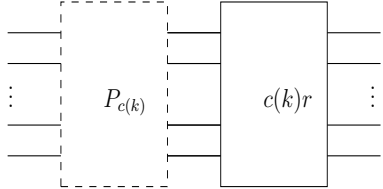}
	\caption{We describe how to form the complex $C^{i,j}_{P,r}(D)$ from a based $c$-colored diagram $D$.  We take the blackboard-framed $c$-cable of $D$ and at the basepoint of a component $k$ of $D$ we tensor in $P_{k(c)}$ and add $c(k)r$ twists as shown in the diagram.  We then take the corresponding cochain complex and shift by $hq^{\sum_k r c(k)(c(k)-1)/2}$.}
	\label{fig:equiv_of_spacify}
\end{figure}
\begin{proof}
We fix j.  We write $C^{i,j}_{P,r}(D)$ to be the cochain complex of free abelian groups formed by following the procedure as outlined in Figure \ref{fig:equiv_of_spacify}.  By the definition of a C-Kffc, there is a framed flow category $\cA^j_{P,r}(D)$ that refines $C^{i,j}_{P,r}(D)$.

Consider the quotient complex $C_1$ of $C^{i,j}_{P,r}(D)$ consisting all generators corresponding to taking the $0$-resolution at each of the crossings of the twist regions at the basepoints.  This corresponds to a downward-closed subcategory $\cA_1$ of $\cA^j_{P,r}(D)$.  
We observe firstly that $|\cA_1|$ is stably homotopy equivalent to $|\cA^j_{P,0}(D)|$ which is exactly $|\cC^j(D^\cab_P)|$, and secondly that the corresponding upward-closed subcategory has trivial cohomology by the turnback-triviality condition on the projectors $P_{c(k)}$.  Hence we have that
\[|\cA^j_{P,r}(D)| \simeq |\cA_1| \simeq |\cC^j(D^\cab_P)| {\rm .}\]

On the other hand, for any value of $r$, the complex $C^{i,j}_{P,r}(D)$ can be written as the total complex
\[  (hq)^{\sum_k rc(k)(c(k)-1)/2} \langle D^{r}_c\rangle  \rightarrow \Gamma^r_1 \rightarrow \cdots \rightarrow \Gamma^r_s \rightarrow \cdots {\rm ,} \]
where each $\Gamma^r_s$ carries an internal differential arising from all crossings of the twisted $c$-cable of $D$, while the part of the differential from $\Gamma^r_s$ to $\Gamma^r_{s+1}$ is induced by the differentials of the $P_{c(k)}$.  
This is because the identity braid smoothing is the only smoothing appearing in cohomological degree zero of each complex $P_{c(k)}$.

Now the minimal quantum degree of a generator in $\bigoplus_r\Gamma^r_s$ tends to $+\infty$ as $s$ tends to $+\infty$ (see Remark \ref{rem:handy_little_remark}).  On the other hand, each $\Gamma^r_s$ is chain-homotopy equivalent by Gauss-elimination to a complex in which the minimal quantum degree is bounded below by $b(r)$, a function independent of $s$ and tending to $+\infty$ as $r$ tends to $+\infty$.  This follows from formula ($4.9$) of \cite{rozansky} and the observation that the cohomological Reidemeister I and II relations can be proved by Gauss-elimination.

Hence the lowest quantum degree of the support of the cohomology of the subcomplex
\[ \Gamma^r_1 \rightarrow \cdots \rightarrow \Gamma^r_s \rightarrow \cdots \]
tends to $+\infty$ as $r$ tends to $+\infty$.  The quotient complex $(hq)^{\sum_k rc(k)(c(k)-1)/2} \langle D^{r}_c\rangle$ corresponds to a downward-closed subcategory of $\cA^j_{P,r}(D)$ with associated stable homotopy type $\bar{\X}^j(D^{r}_c)$.  So for large enough $r$ we have
\[ \Xcol^j(D_c) \simeq \bar{\X}^j(D^{r}_c) \simeq |\cA^j_{P,r}(D)| \simeq |\cA_1| \simeq |\cC^j(D^\cab_P)| {\rm .} \]
\end{proof}

\begin{remark}
	We have worked here with colored links, but all of what we have done applies, \emph{mutatis mutandis}, to more general (in other words, not just diagrams obtained by cabling) closed diagrams containing Jones-Wenzl projectors.
\end{remark}

\section{Lifting the Cooper-Krushkal projectors}
\label{sec:CK_spacify_construction}
In this section we give a C-Kffc associated to link diagrams colored with colors drawn from the palette $\{1, 2\}$.  It would seem \emph{a priori} very likely that the methods used in this construction should extend to the color $3$, since for this color we have (due to Cooper-Krushkal \cite{cooperkrushkal}) an explicit and fairly simple cohomological projector.  However, it turns out that there is an unexpected non-trivial obstruction to this extension.  
The obstruction can be obviated by renormalizing the 3-colored Jones invariant of the $0$-framed unknot to be
\[ (q^{-2} + 1 + q^2)(1 - q^2 + q^4 - q^6 + \cdots) \,\,\,\, {\rm rather} \,\, {\rm than} \,\,\,\, q^{-2} + 1 + q^2 {\rm .}\]
We briefly discuss the obstruction and renormalization at the end of Section \ref{sec:examples}, but we do not give in this paper the full construction of the renormalized spacification.

\subsection{A 2-colored Cooper-Krushkal projector}
\label{subsec:CK2}
In \cite{JLS}, two of the authors together with Dan Jones considered the 2-stranded braid of $k$ crossings, each of the same sign.  The Bar-Natan complex of this tangle has a particularly simple form: it is homotopy equivalent to a complex which has one circleless smoothing in each cohomological degree from $-k/2$ to $k/2$ (with the grading conventions used in this paper).  
Indeed, in Figure \ref{fig:ck2},
we give the Cooper-Krushkal projector for the color $2$; the Bar-Natan complex for the positively twisted $k$-crossing $2$-braid is, up to an overall shift, the quotient complex of this projector consisting all tangles of cohomological degree less than $k+1$.

Decomposing a closed link diagram $D$ into a tensor product of such tangles one can consider the tensor product of their simplified chain homotopy class
representives.  This gives a cochain complex $\langle D \rangle^{\rm simp}$ (depending on the decomposition of $D$) of free abelian groups, and $\langle D \rangle^{\rm simp}$ is refined by a framed flow category given in \cite{JLS}.  
The associated stable homotopy type was shown to be independent of the choice of decomposition, and it was observed that the decomposition in which each tangle has a single crossing returns the Lipshitz-Sarkar framed flow category.

\begin{figure}
	\resizebox{10cm}{1cm}{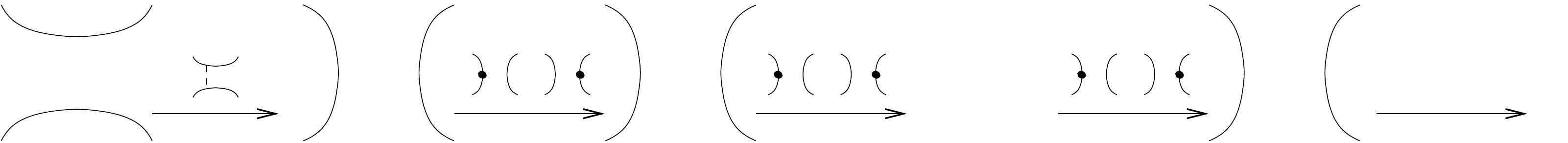}
	\caption{We show here the Cooper-Krushkal projector.  We suppress the degree-shifts for ease of visualization.  The degree shifts can be determined by noting that the identity-braid or horizontal smoothing on the far left is in cohomological degree $0$ and quantum degree $0$, and all differentials raise the cohomological degree by $1$ and preserve the quantum degree.}
	\label{fig:ck2}
\end{figure}

Taking a suitably normalized version of this construction for $k=\infty$ gives a construction of a C-Kffc.  In particular, this construction enables us to make non-trivial calculations of the colored stable homotopy types of the $(2,1)$-colored Hopf link as well as of the $2$-colored trefoil.

Suppose that $T$ is a tangle diagram in the plane punctured by $k$ discs each with $4$ ordered boundary points.  Let the closed diagram $T^r$ be given by filling in each disc with $2r$ positive twists.

We consider a particular decomposition of $T^r$ into a tensor product of tangles - specifically, we take one tangle (of $2r$ crossings) at each filled disc, one tangle for every other crossing of $T^r$, and finally the rest of the diagram which is crossingless.

Such a decomposition into tangles is exactly the input into the construction of the paper \cite{JLS}.  So, incorporating now an overall shift and fixing a quantum degree $j$, there is a framed flow category $\cA^j(T^r)$ refining the quantum degree $j$ part of the simplified cochain complex $(hq)^{kr}\langle T^r \rangle^{\rm simp}$.

Finally we note that for fixed $j$ and large enough $r$, the quantum degree $j$ part of $(hq)^{kr}\langle T^r \rangle^{\rm simp}$ agrees with the quantum degree $j$ part of the Cooper-Krushkal complex $\langle T_P \rangle$.  So, taking $r$ to be large, the framed flow category $\cA^j(T^r)$ provides our candidate for a C-Kffc.  The remaining properties required of a C-Kffc are now straightforward to verify.

\section{Examples}
\label{sec:examples}

\subsection{The $2$-colored unknot}

Consider a diagram of the blackboard framed $2$-cable of the $0$-crossing unknot containing a Cooper-Krushkal projector $P_2$. The generators in the resulting cochain complex come from smoothings with two circles in homological degree $0$, and one circle in homological degree bigger than $0$, compare Figure \ref{fig:ck2}. The minimal quantum degree in which we get a generator is therefore $q=-2$ with one generator in homological degree $0$. 
For $q=0$ we get two generators in homological degree $0$ and one in homological degree $1$.
For $q=2$ there is one generator in homological degrees $0$, $1$ and $2$ each.

For $q=2j$ with $j\geq 2$ we get two generators, one in homological degree $j-1$ and one in degree $j$. The coboundary map alternates between multiplication by  $0$ and $2$. The cohomology is therefore easily calculated, and determines the stable homotopy types because of thinness. We thus get
\begin{align*}
 \Xcol^{-2}(U_2) &= S^0 \hspace{1cm}
 \Xcol^{0}(U_2) = S^0 \hspace{1cm}
 \Xcol^{2}(U_2) = S^2 \\
 \Xcol^{4j}(U_2) &= M(\Z/2,2j)\mbox{ for }j\geq 1 \\
 \Xcol^{4j+2}(U_2) &= S^{2j+1}\vee S^{2j+2} \mbox{ for }j \geq 1.
\end{align*}
Note that the notation $M(G,n)$ stands for a Moore space, a space whose only non-trivial integral homology group is $G$ in degree $n$.

\subsection{The $2$-colored trefoil}
\begin{figure}
	\resizebox{5cm}{5cm}{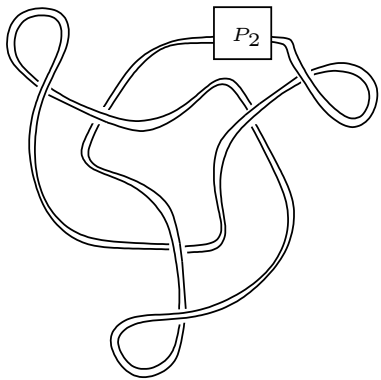}
	\caption{We show the $0$-framed $2$-cable of the right-handed trefoil with a Cooper-Krushkal projector placed on it.}
	\label{fig:framed3_1}
\end{figure}

In Figure \ref{fig:framed3_1} we give a diagram of a $2$-cable of the right-handed trefoil $T$ containing a Cooper-Krushkal projector $P_2$.  The extra loops ensure that we get the $0$-framed $2$-cable, and we denote it by $T^0_2$. For each quantum degree $j$ this diagram gives rise to a framed flow category $\cA$ as described in Subsection \ref{subsec:CK2}.

For calculational purposes, we want to remove the three double loops. Performing two Reidemeister I moves and one Reidemeister III move turns each double loop into a $(-2)$-tangle, which can be absorbed by the projector $P_2$. However, because of the Reidemeister I moves, we get a shift in homological and quantum degrees. More precisely, we get $\langle D_2^r\rangle = h^3q^9\langle D'_2\,\!^{r-3}\rangle$, where $D'$ is the standard $3$-crossing diagram of the right-handed trefoil. Denoting the 2-colored right-hand trefoil with framing $3$ by $T^3_2$, we get $\Khcol^{i,j}(T^0_2)=\Khcol^{i+6,j+12}(T^3_2)$.

Taking these shifts into account and working with the diagram for $T^3_2$, we see that the least quantum degree in $\cA=\cA_0$ which admits an object is given by $q=2$ with homological degree $h=0$, coming from a smoothened diagram with $4$ circles. This is indeed the only object in this quantum degree.

The projector $P_2$ gives rise to upward closed subcategories $\cA_k$ for $k\geq 0$ generated by objects that arise from a tangle in $P_2$ of cohomological degree at least $k$. The highest quantum degree of an object in $\cA_0-\cA_1$ is $q=24$ coming from $6$ circles in the smoothened diagram. It follows that for quantum degree $q\geq 26$ the relevant flow category $\cA^q$ is a full subcategory of $\cA_1$.

The quotient category $\cA_j/\cA_{j+1}$ for $j\geq 1$ is, up to degree shifts, the Lipshitz-Sarkar flow category of a diagram of the unknot with $12$ crossings. Furthermore, this diagram can be transformed into the standard unknot diagram by performing six Reidemeister II moves. The category $\cA_j/\cA_{j+1}$ for $j\geq 1$ is therefore stably equivalent to a flow category containing two objects of homological degree $j+6$, one of quantum degree $2j+12$, the other of quantum degree $2j+10$.

Also notice that the associated cochain complexes to the flow categories $\cA^q$ and $\cA^{q+4}$ for $q\geq 26$ only differ in a cohomological shift by $2$.  If the tail turns out to be cohomologically thin (as it does), it follows that the stable homotopy types for $q$ up to $28$ determine all the stable homotopy types.  The stable homotopy types for $q$ up to $28$ may be determined using the diagram $D'_2\,\!^r$ for large $r$.  It turns out that $r=8$ is sufficient, and the following calculations have been done using the programme \verb+KnotJob+ available at 
\href{http://www.maths.dur.ac.uk/~dma0ds/knotjob.html}{http://www.maths.dur.ac.uk/$\sim$dma0ds/knotjob.html}.

We can identify all stable homotopy types from cohomology and Steenrod square calculations using the classification result of Baues-Hennes \cite{bauhen} with the exception of $q=10$, where $\X^{10}_{{\rm col}(2)}(T)$ is either $S^{3}\vee S^{4} \vee S^6$ or $X(\varepsilon,3)\vee S^{4}$. Recall that $X(\varepsilon,n)$ is the space obtained by attaching an $(n+3)$-cell to $S^n$ using the nontrivial element of $\pi_2^{\mathbf{st}}\cong \Z/2$. Excluding this, we get
\[
 \begin{array}{ll}
  \Xcol^{2}(T^0_2) = S^{0} & \Xcol^{4}(T^0_2) = S^{0} \\
  \Xcol^{6}(T^0_2) = S^{2} & \Xcol^{8}(T^0_2) = X(_2\eta,2) \\
  \Xcol^{12}(T^0_2) = X(\eta2,5)\vee S^6 & \Xcol^{14}(T^0_2) = X(\eta2,5)\vee S^7 \vee S^8 \vee S^8 \\
  \Xcol^{16}(T^0_2) = S^7 \vee M(\Z/4,8) \vee M(\Z/2,8) & \Xcol^{18}(T^0_2) = S^9 \! \vee \! S^9 \! \vee \! M(\Z/2,9)\! \vee \! S^{10} \\
  \Xcol^{20}(T^0_2) = M(\Z/2,9)\vee M(\Z/2,10)\vee S^{11} & \Xcol^{22}(T^0_2) = S^{11}\vee M(\Z/2,11) \vee S^{12} \\
  \Xcol^{24}(T^0_2) = S^{12} \vee M(\Z/2,12)
 \end{array}
\]
The tail is given by
\begin{align*}
 \Xcol^{4j+2}(T^0_2) &= S^{2j+1}\vee S^{2j+2} \mbox{ for }j \geq 6 \\
 \Xcol^{4j}(T^0_2) &= M(\Z/2,2j)\mbox{ for }j\geq 7.
\end{align*}
Notice that for $j\geq 26$ we have $\Xcol^{j}(T^0_2)=\Xcol^j(U_2)$.

The notation $X(\eta2,n)$ is taken from \cite{bauhen}, and stands for an elementary Chang complex. It is an appropriate suspension of $\RP^4/\RP^1$ such that the first non-trivial homology group is in degree $n$. Similarly, $X(_2\eta,m)$ is a suspension of $\RP^5/\RP^2$ such that the first non-trivial homology group is in degree $m$. Both spaces have non-trivial $\Sq^2$ and are therefore not wedges of Moore spaces.

%
%
%
%

\subsection{The $(2,1)$-colored Hopf link}

\begin{figure}
	\resizebox{5cm}{5cm}{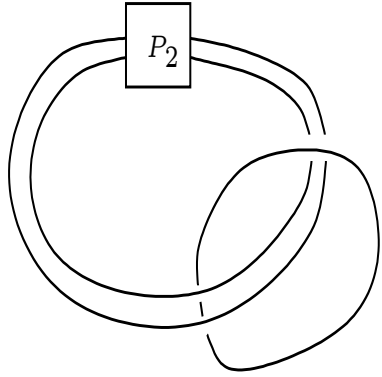}
	\caption{This is the $0$-framed $(2,1)$-cable of the Hopf link in which the $2$-cabled component receives a Cooper-Krushkal projector.}
	\label{fig:hopf21}
\end{figure}

We denote the $(2,1)$-colored $0$-framed Hopf link by $H_{2,1}$.  In Figure \ref{fig:hopf21} we give a diagram of the Hopf link in which one of the components has been replaced by a $0$-framed $2$-cable containing a Cooper-Krushkal projector $P_2$.  For each quantum degree $j$ this diagram gives rise to a framed flow category as described in Subsection \ref{subsec:CK2}.  The associated stable homotopy type is $\Xcol(H_{2,1})$.

Note that the diagram consists of the tensor product of three parts: the projector $P_2$, and then two tangles each of which is a $2$-crossing $2$-braid.  
As before, we can filter the flow category via the projector, leading to categories $\cA_j$ for $j\geq 0$.


For actual calculations, we replace the projector with a $(2r)$-tangle, so the resulting diagram is that of the $P(-2,2,2r)$ pretzel link. For a given quantum degree we can then use the method of \cite{JLS} to get a flow category built from three tangles. The lowest quantum degree for which we can get an object is $q=-5$, for which there is exactly one object of homological degree $-2$. 

For $q\geq 7$, all objects are contained in $\cA_1$, and the categories $\cA^{2j-1}$ and $\cA^{2j+3}$ for $j\geq 4$ have the following similarity.
If $\alpha$ is an object of $\cA^{2j-1}$ which also sits in $\cA_k$ for $k\geq 1$, there is a corresponding object $\bar{\alpha}$ in $\cA^{2j+3}$ also in $\cA_{k+2}$ with $|\bar{\alpha}|=|\alpha|+2$. It is clear from the framing formulas in \cite{JLS} that for the moduli spaces $\mathcal{M}(\alpha,\beta)\cong \mathcal{M}(\bar{\alpha},\bar{\beta})$ as framed manifolds, provided these are at most $1$-dimensional. 

Therefore the colored Khovanov cohomology of the tail is periodic, and since we only get non-trivial cohomology groups in three adjacent degrees, we also get periodicity of the stable homotopy type in the tail. This uses that the $1$-dimensional moduli spaces agree with framing for $\cA^{2j-1}$ and $\cA^{2j+3}$.
Calculation of Khovanov cohomology and the second Steenrod Square shows that
\begin{align*}
 \Xcol^{-5}(H_{2,1}) &= S^{-2} & \Xcol^{-3}(H_{2,1}) &= S^{-2} \\
 \Xcol^{-1}(H_{2,1}) &= S^0 & \Xcol^{1}(H_{2,1}) &= X(_2\eta,0) \\
 \Xcol^{3}(H_{2,1}) &= S^1 \vee S^2 \vee S^2 & \Xcol^{5}(H_{2,1}) &= X(_2\eta,2) 
\end{align*}
The tail is given by
\begin{align*}
 \Xcol^{4j-1}(H_{2,1}) &=X(\eta 2,2j-1)\vee S^{2j} \mbox{ for }j\geq 2 \\
 \Xcol^{4j+1}(H_{2,1}) &= X(_2\eta,2j)\vee S^{2j+1}  \mbox{ for }j\geq 2.
\end{align*}
The $(2,1)$-colored unlink $U_{2,1}$ is the disjoint union of the $1$-colored $0$-framed unknot $U_1$ and the $2$-colored $0$-framed unknot $U_2$. The stable homotopy type can therefore be derived using \cite[Thm.1]{LawLipSar}. More precisely, we get
\[
 \Xcol^j(U_{2,1}) = (\X^1(U) \wedge \Xcol^{j-1}(U_2)) \vee (\X^{-1}(U) \wedge \Xcol^{j+1}(U_2)).
\]
Since both $\X^1(U)=S^0=\X^{-1}(U)$, we have that $\Xcol^j(U_{2,1})$ is a wedge of Moore spaces for all $j$. In the tail we have
\begin{align*}
 \Xcol^{4j-1}(U_{2,1}) &= S^{2j-1}\vee S^{2j}\vee M(\Z/2,2j) \mbox{ for }j\geq 2, \\
 \Xcol^{4j+1}(U_{2,1}) &= M(\Z/2,2j) \vee S^{2j+1} \vee S^{2j+2} \mbox{ for }j\geq 2.
\end{align*}
In particular, we have
\[
 \Khcol^{i,j}(U_{2,1}) = \Khcol^{i,j}(H_{2,1})
\]
for all $j\geq 7$ (a result that for high enough $j$ is not unexpected, and that can be derived in ways other than brute calculation), but
\[
 \Xcol^j(U_{2,1}) \not\simeq \Xcol^j(H_{2,1}) {\rm .}
\]

\subsection{A conjecture on the $3$-colored unknot}

The stable homotopy type of the $0$-framed $3$-colored unknot $\X^j_{\rm col}(U_3)$ was partially computed by Willis \cite{Willis}, who showed that it was not a wedge of Moore spaces and so, in some sense, more interesting than just the colored Khovanov cohomology.

The $3$-colored Khovanov cohomology can easily be calculated from \cite[\S 4.4]{cooperkrushkal}. We summarize this in Table \ref{tab:3-unknot}.
\begin{table}
\setlength\extrarowheight{3pt}
\begin{tabular}{|l|c|c|c|c|}
 \hline
 & $i=1$ & $i=2$ & $i=3$ & $i=4$ \\
 \hline
 $\Khcol^{i-4,-3}(U_3)$ & & & & $\Z$ \\
 \hline
 $\Khcol^{i-4,-1}(U_3)$ & & & & $\Z$ \\
 \hline
 $\Khcol^{i,1}(U_3)$ & & $\Z$ & &  \\
 \hline
 $\Khcol^{i,3}(U_3)$ & & & $\Z/2$ & $\Z$ \\
 \hline
 $\Khcol^{i+4j,6j+5}(U_3),\, j\geq 0$ & & & $\Z$ & $\Z$ \\
 \hline
 $\Khcol^{i+4j,6j+1}(U_3),\, j \geq 1$ & $\Z$ & $\Z$ & & \\
 \hline
 $\Khcol^{i+4j,6j+3}(U_3),\, j \geq 1$ & $\Z$ & & $\Z/2$ & $\Z$ \\
 \hline
\end{tabular}
\vspace{0.3cm}
\label{tab:3-unknot}
\caption{The $3$-colored Khovanov cohomology of the unknot.}
\end{table}

We observe that the tail is $3$-periodic in quantum degrees $q=2j+1$ starting from $j\geq 2$ with a homological shift by $4$. Also, by simply looking at these groups we see that except for quantum degrees $q=6j+3$ with $j\geq 0$ the stable homotopy types are wedges of spheres. In quantum degree $q=3$ we have the non-trivial Steenrod Square coming from the torus knot $T_{4,3}$ first observed in \cite{LipSarSq}, and which stably survives by \cite{Willis}.

The quantum degree $q=9$ can be realized by the torus knot $T_{7,3}$, and computer calculations show a non-trivial $\Sq^2$ in degree $5$, with $\Sq^2$ trivial in degree $6$. The triviality in degree $6$ indicates that the tail of the stable homotopy types is not $3$-periodic, as the difference in $3$-colored Khovanov cohomology in quantum degrees $q=3$ and $q=9$ comes from an extra generator in homological degree $0$ killing the cocycle in degree $1$, which survives in degree $5$ for $q=9$.

Computer calculations on $T_{8,3}$ show a trivial $\Sq^2$ in degree $9$, although this is not yet in the stable range for $q=15$. Using a suitable diagram with a low number of tangles we have made computer calculations for $T_{13,3}$ which give evidence for the conjecture below.

\begin{conjecture}
 For $j \geq 1$ we have
\begin{align*}
 \Xcol^{12j-3}(U_3) &= X(\eta2,8j-3) \vee S^{8j}, \\
 \Xcol^{12j+3}(U_3) &= S^{8j+1} \vee X(_2\eta,8j+2).
\end{align*}

\end{conjecture}

Note that these two spaces are not stably homotopy equivalent, although they are Spanier-Whitehead dual when appropriately shifted. Following consideration of the Cooper-Krushkal projector $P_3$ explicitly described in \cite{cooperkrushkal} this conjecture is somewhat surprising. From $P_3$ the $3$-periodicity follows immediately, so one may expect the same periodicity in the tail of the stable homotopy type.

Indeed, if one attempts a spacification based on lifting the Cooper-Krushkal projector $P_3$ to a framed flow category, one finds that the natural first attempt gives rise to $1$-dimensional moduli spaces the framings of which also follow $3$-periodicity.  However, if one then pushes a little further to determine if one can genuinely lift $P_3$ to a C-Kffc, one runs into `ladybug matching' type problems which cannot all be solved simultaneously in a natural way, at least to the authors' eyes.

On the other hand, suppose that $D$ is a tangle diagram in a disc with $6$ ordered boundary points.  This gives a cochain complex of free abelian groups $\langle D_P \rangle$.  Now consider the `reduced' subcomplex $\langle D_P \rangle^{\rm red}$ of $\langle D_P \rangle$ obtained by restricting to half the generators of $\langle D_P \rangle$.  Specifically, restrict to only those generators arising from a decoration by $v_-$ of a chosen boundary point of $D$.  In such a situation one can lift the cochain complex $\langle D_P \rangle^{\rm red}$ to a framed flow category refining it.  The ladybug matching problems no longer occur since we have thrown out enough objects of the flow category to kill them.

Unfortunately, this subcomplex is not really a very natural object to consider.  The graded Euler characteristic is a renormalized version of the $3$-colored Reshetikhin-Turaev invariant as discussed at the start of Section \ref{sec:CK_spacify_construction}, but it is hard to motivate why one should consider this renormalization.  Therefore we do not pursue this further here.

\bibliographystyle{amsplain}
\bibliography{References.bib}
\end{document}